\documentclass[12pt,a4]{article}
\usepackage{amsfonts}
\usepackage{amssymb}
\usepackage{amsmath}
\usepackage{graphicx}
\newtheorem{proposition}{Proposition}[section]

\newtheorem{definition}{Definition}[section]
\newtheorem{corollary}{Corollary}[section]
\newenvironment{proof}{{\bf Proof:} }{}
\begin{document}

\title{{\Large Generalized Poincar\'{e} Half-Planes}}
\author{{\normalsize R\"{u}stem Kaya} \\
\textit{Eski\c{s}ehir Osmangazi University,}\\
\textit{Department of Mathematics-Computer,\ }\\
\textit{26480 Eski\c{s}ehir, Turkey.}\\
rkaya@ogu.edu.tr}
\maketitle

\begin{abstract}
In this note, we give some generalisations of the classical Poincar\'{e}
upper half-plane, which is the most popular model of hyperbolic plane
geometry. For this, we replace the circular arcs by elliptical arcs with
center on the $x-$axis, and foci on the $x-$axis or on the lines
perpendicular to the $x-$axis at the center, in the upper half-plane. Thus,
we obtain a class of generalized upper half-planes with infinite number of
members.

Furthermore we show that every generalized Poincar\'{e} upper half-plane
geometry is a neutral geometry satisfying the hyperbolic axiom. That is, it
satisfies also all axioms of the Euclidean plane geometry except the
parallelism.


\textit{Ams Subject classification }:\ 51F05, 51K05
\end{abstract}

\section{Introduction}

The concept of upper half-plane is used to mean the Cartesian half-plane
which consists of all points with positive ordinate. If half-lines which are
perpendicular to the $x-$axis, and half-circles with center on the $x-$axis
are defined as lines in the upper half-plane, one gets a model for the
hyperbolic plane. (For the other models of the hyperbolic plane geometries
see [1], [4], [7], [8].) In this model, the hyperbolic length of an
arbitrary curve $\gamma $ is defined by%
\begin{equation*}
\int_{\gamma }\frac{\sqrt{dx^{2}+dy^{2}}}{y}
\end{equation*}%
which reduces also a distance function and a metric known as Poincar\'{e}
metric. The Euclidean half-plane together with Poincar\'{e} metric is
generally known as the Poincar\'{e} half plane. (During the recent years
many metric models have also been developed see [2], [3], [5], [6].)

In this note we define and examine some generalisations of the Poincar\'{e}
Half-Plane.

\section{Basic Conceps and Definitions}

\begin{definition}
\textit{Let }$\mathcal{P}$ be the set of all Cartesian points on the upper
half plane, that is 
\begin{equation*}
\mathcal{P}\text{ }=\left\{ (x,y)\mid x,y\in \mathbb{R}\text{ , }y>0\right\}
,
\end{equation*}%
and let 
\begin{equation*}
\mathcal{L}_{p}\text{ }=\left\{ L_{p}\mid p\in \mathbb{R}\right\} \text{
such that }L_{p}\text{ }=\left\{ (p,y)\in \mathcal{P}\right\} ;
\end{equation*}
that is, $L_{p}$ represent the Euclidean half line with the equation $x=p,$ $%
y>0$.

Let 
\begin{equation*}
\mathcal{L}_{kac}\text{ }=\left\{ L_{kac}\mid a,c,k\in \mathbb{R}\text{ , }%
a>0,k>0,k\text{ constant}\right\}
\end{equation*}%
such that%
\begin{equation*}
L_{kac}\text{ }=\left\{ (x,y)\in \mathcal{P}\mid
(x-c)^{2}+k^{2}y^{2}=a^{2}\right\} .
\end{equation*}%
That is, $L_{kac}$ is an Euclidean half-ellipse with center $(c,0)$ on the $%
x-$axis. Where positive real number $k$ is a given constant. $a$ represents
the length of the semimajor axis or semiminor axis of the ellipse according
as $k>1$ or $0<k<1$ . $a$ is always mesured along the $x-$axis. If $b$ is
the length of the other \ axis than $a=bk$.

Now, define 
\begin{equation*}
\mathcal{L=L}_{p}\cup \mathcal{L}_{kac}\text{.}
\end{equation*}%
Elements of $\mathcal{L}$ are called \textbf{hyperbolic lines}, shortly $h-$%
\textbf{lines}. Now consider the system%
\begin{equation*}
\mathbb{H}_{k}=\left\{ \mathcal{P}\text{ ,}\mathcal{L}\right\}
\end{equation*}%
which is called a \textbf{generalized Poincar\'{e} half} \textbf{plane}.
Notice that, every positive real number $k$ determines a generalized Poincar%
\'{e} half plane. Thus, now we have a family of generalized Poincar\'{e}
half-planes with infinite number of members. If \ $k=1$ then $\mathbb{H}_{1}$
is the the classical Poincar\'{e} half plane.%

\begin{figure}[ht]
	\centering
	\includegraphics[width=1\textwidth]{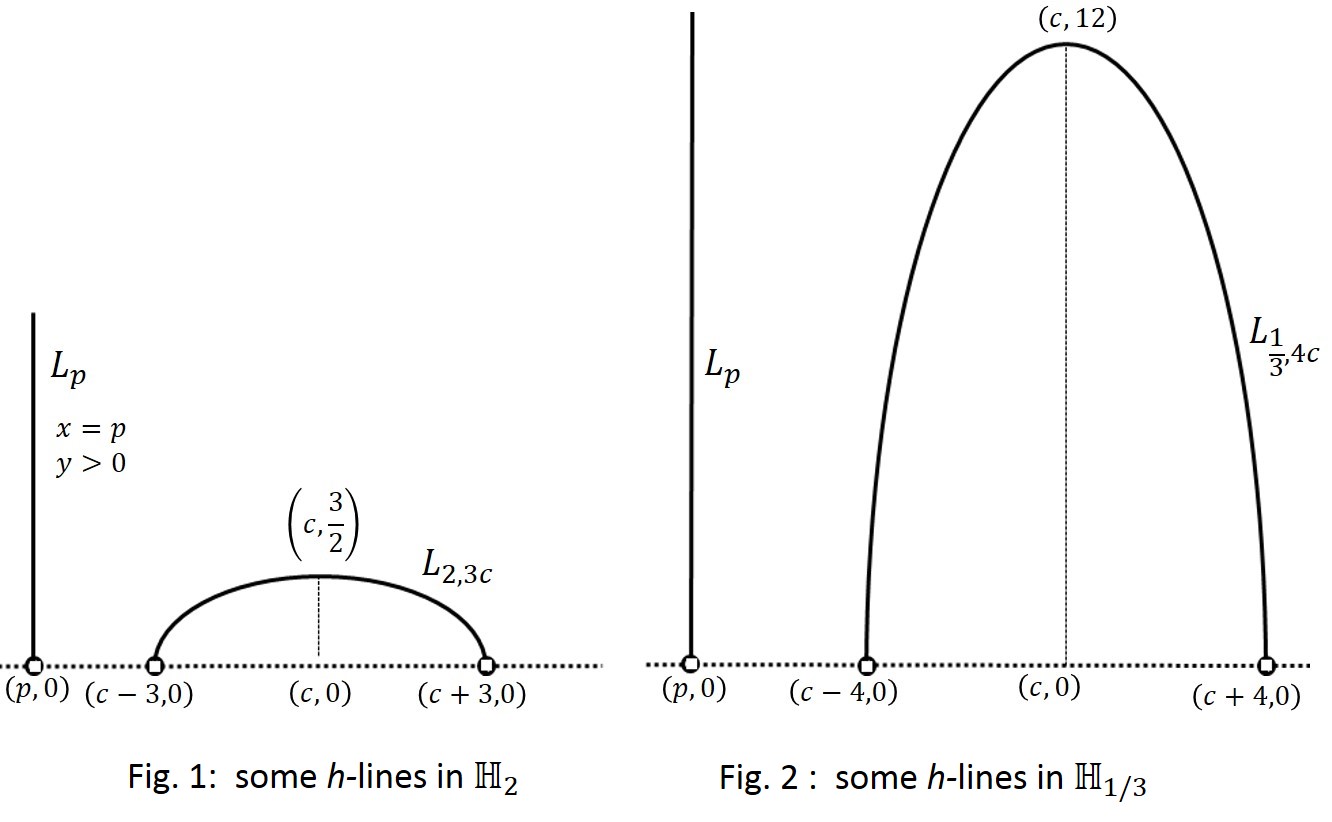}	
\end{figure}

Notice that if $k>1$ then $k$ is ratio of the length of semimajor axis to
the length of semiminor axis of the ellipse. In this case the $x-$axis is
the major axis.

if $\ 0<k<1$ then $k$ is ratio of the length of semiminor axis to the length
of semimajor axis of the ellipse. In this case the the major axis is
perpendicular to the $x-$axis at $(c,0)$.
\end{definition}

\begin{proposition}
There exist a unique $h-$line pasing through two distinct points of $%
\mathcal{P}$ in $\mathbb{H}_{k}$.
\end{proposition}

\begin{proof}
Let $P_{1}=(x_{1},y_{1})$ and $P_{2}=(x_{2},y_{2})\in \mathcal{P}$. If $%
x_{1}=x_{2}$ then clearly $P_{1}P_{2}$ is the $h-$line $L_{x_{1}}$ with
equation $x=x_{1},$ $y>0$ by definition of $L_{p}$. And obviously, there is
no $h-$line of type $L_{kac}$ passing through $P_{1}$ and $P_{2}$.

If $x_{1}\neq x_{2}$ and $P_{1},P_{2}\in L_{kac}$ then 
\begin{equation*}
\begin{array}{c}
(x_{1}-c)^{2}+k^{2}y_{1}^{2}=a^{2} \\ 
(x_{2}-c)^{2}+k^{2}y_{2}^{2}=a^{2}\text{.}%
\end{array}%
\end{equation*}%
Solving this system of equations for $a$ and $c$ one obtains 
\begin{equation*}
c=\frac{(x_{1}-x_{2})^{2}+k^{2}(y_{1}^{2}-y_{2}^{2})}{2(x_{1}-x_{2})}
\end{equation*}%
and%
\begin{equation*}
a=[(x_{1}-c)^{2}+k^{2}y_{1}^{2}]^{1/2}\text{.}
\end{equation*}%
Since $k$ is costant there exist a unique pair $a$ and $c$ and consequently one obtains a unique $h-$line $L_{kac}$. Obviously there is no $h-$%
line of type $L_{p}$ having on such a pair of points.
\end{proof}

\begin{corollary}
Two hyperbolic lines meet at most one point in $\mathbb{H}_{k}$.
\end{corollary}

Notice that althought the $x-$axis is not in $\mathbb{H}_{k}$, we can use
the points on it in the definitions and calculations as follows:

\begin{definition}
Every pair of two $h-$lines which are Euclidean half-lines are defined as%
\textbf{\ parallel} $h-$\textbf{lines}. Also, two different $h-$lines are
called \textbf{parallel} iff their Euclidean extensions meet on the $x-$axis.

Thus, $L_{p}//L_{q}$ for all $p\neq q$,

$L_{p}//L_{kac}$ $\Longleftrightarrow L_{p}\cap L_{kac}\in \left\{
(c-a,0),(c+a,0)\right\} $

$L_{kac}//L_{ka^{\prime }c^{\prime }}$ $\Longleftrightarrow L_{kac}\cap
L_{ka^{\prime }c^{\prime }}\in \left\{ (c-a,0),(c+a,0)\right\} $.
\end{definition}

\begin{proposition}
$\mathbb{H}_{k}$ satisfies the hyperbolic property, that is each $h-$line $L$
and each point $P\notin L$ there exist exactly two hyperbolic lines through $%
P$ and parallel to $L$.
\end{proposition}

\begin{proof}
Proof can be easily given using the definitions and proposition 2.1.
\end{proof}

If one defines that \textit{two hyperbolic lines are parallel when they are
disjoint}, then, clearly, there exist infinitely many different hyperbolic
lines through $P$ that are parallel to $L$.

\section{Further Properties of $\mathbb{H}_{k}$}

\bigskip As it is well known an\textit{\ incidence geoemetry }is geometry $I$%
, consist of a set $\mathcal{P}$, whose elements are called points, together
with a collection $\mathcal{L}$ of non-empty subsets of $\mathcal{P}$,
called lines, such that:

A1) Every two distinct points in $\mathcal{P}$ lies on a unique line,

A2) There exist three points in $\mathcal{P}$, which do not lie all on a
line.

An incidence geometry is a \textit{metric geometry} if

A3) There exists a distance function 
\begin{equation*}
d:\mathcal{P\times P\rightarrow }\mathbb{R}\text{ }\ni \text{ for all }%
P,Q\in \mathcal{P}\text{ such that }
\end{equation*}%
i) $d(P,Q)\geq 0$; ii) $d(P,Q)=0$ iff $P=Q$ and iii) $d(P,Q)=d(Q,P)$; and

A4) There exists a one-to-one, onto function $f:l\rightarrow $ $\mathbb{R},$ 
$\forall l\in \mathcal{L}$ such that $\left\vert f(P)-f(Q)\right\vert =d(P,Q)
$ for each pair of points $P$ and $Q$ on $l$ (\textit{Ruler postulate}.)

Clearly, every $\mathbb{H}_{k}$ has an infinite number of points and lines
and satisfies axioms of the incidence geometry.

Now the question is that \textit{whether the every }$\mathbb{H}_{k}$\textit{%
\ is a metric geometry or not}. If it is a metric geometry what are its
distance function and its ruler $f$.

It is known that if $P=(x_{1},y_{1})$ and $Q=(x_{2},y_{2})$ are points in
the Poincar\'{e} plane, the distance function is given by 
\begin{equation*}
d(P,Q)=\left\{ 
\begin{array}{l}
\left\vert Ln\left( \dfrac{x_{1}-c+r}{y_{1}}\text{ }/\text{ }\dfrac{%
x_{2}-c+r}{y_{2}}\right) \right\vert \text{ if }x_{1}\neq x_{2} \\ 
\\ 
\left\vert L{n}\left( {y_{2}}/{y_{1}}\right) \right\vert \text{ \ \ \ \
\ \ \ \ \ \ \ \ \ \ \ \ \ \ \ \ \ \ \ \ \ \ \ \ if }x_{1}=x_{2}\text{.}%
\end{array}%
\right\} 
\end{equation*}%
and the ruler $f$ is given by%
\begin{equation*}
f(x,y)=\left\{ 
\begin{array}{l}
\left\vert L{n}(\dfrac{x-c+r}{y})\right\vert \text{ if }(x,y)\in \mathcal{C}
\\ 
\left\vert L{n}\text{ }y\right\vert \text{ \ \ \ \ \ \ \ \ \ \ \ \ \ \ if }%
(x,y)\in L_{p}\text{.}%
\end{array}%
\right\} 
\end{equation*}%
where $\mathcal{C}$ stands for the semi-circle 
\begin{equation*}
(x-c)^{2}+y^{2}=r^{2},\text{ }y>0\text{.}
\end{equation*}%
Now, we will use the above $d$ and $f$ to give a reasonable ruler and a
distance function for $\mathbb{H}_{k}$. For this, consider the
transformation 
\begin{equation*}
g:\mathcal{P}\text{ }\rightarrow \mathcal{P}\text{ }\ni \text{ }g(x,y)=(x,ky)
\end{equation*}%
which maps $L_{kac}$ to the above semicircle $\mathcal{C}$ and the line $%
L_{p}$ to itself one-to-one onto (see Fig.3). 

\begin{figure}[ht]
	\centering
		\includegraphics[width=1\textwidth]{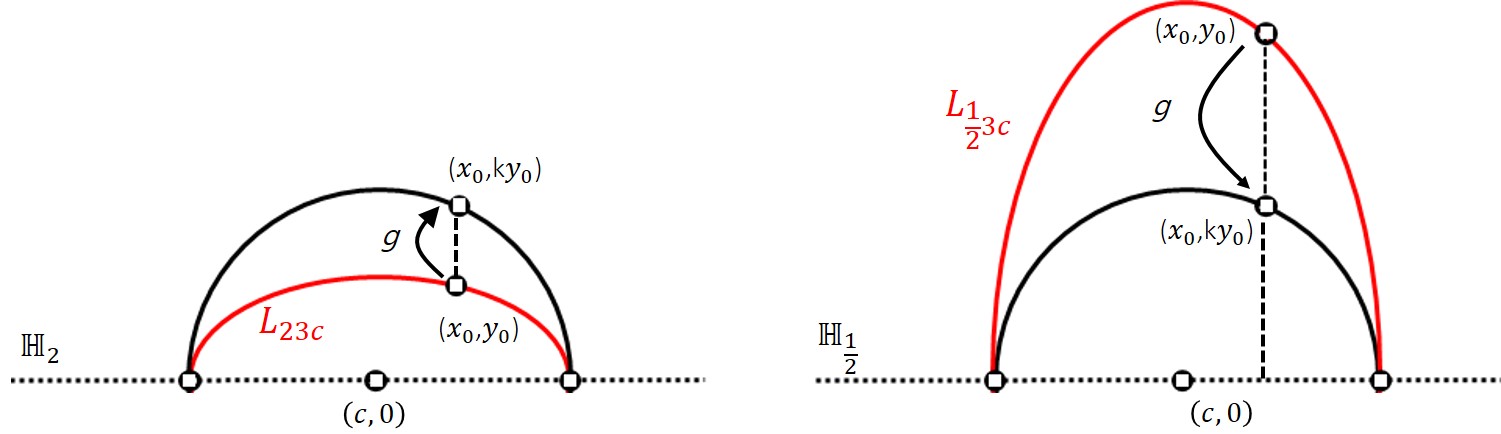}
	\text{ Fig.3: $ g:(x,y) \rightarrow (x,ky) $ }
\end{figure}

Let us define distance function for $P_{1}=(x_{1},y_{1})$ and $%
Q=(x_{2},y_{2})$ in $\mathbb{H}_{k}$ as 
\begin{equation*}
\begin{array}{ll}
d_{k}(P,Q) & :=d(g(P),g(Q))=d((x_{1},ky_{1}),(x_{2},ky_{2})) \\ 
&  \\ 
& =\left\{ 
\begin{array}{l}
\left\vert L{n}\left( \dfrac{(x_{1}-c+a)}{ky_{1}}\text{ }/\text{ }(\dfrac{%
(x_{2}-c+a)}{ky_{2}}\right) \right\vert \text{ \ \ \ \ if \ \ \ \ }x_{1}\neq
x_{2} \\ 
\left\vert L{n}\text{ }\left( ky_{2}\text{ }/\text{ }ky_{1}\right)
\right\vert \text{ \ \ \ \ \ \ \ \ \ \ \ \ \ \ \ \ \ \ \ \ \ \ \ \ \ \ \ \ \
\ \ if \ \ \ \ \ }x_{1}=x_{2}%
\end{array}%
\right.  \\ 
&  \\ 
& =\left\{ 
\begin{array}{l}
\left\vert L{n}\left( \dfrac{(x_{1}-c+a)y_{2}}{(x_{2}-c+a)y_{1}}\text{ }%
\right) \right\vert \text{ \ \ \ \ \ \ if }x_{1}\neq x_{2} \\ 
\left\vert L{n}\text{ }\left( y_{2}\text{ }/\text{ }y_{1}\right)
\right\vert \text{ \ \ \ \ \ \ \ \ \ \ \ \ \ \  \ \ \ \ \ if \ }x_{1}=x_{2}%
\end{array}%
\right.  \\ 
& =d(P,Q)\text{.}%
\end{array}%
\end{equation*}%
And define the ruler $f_{k}$ as $f_{k}=f\circ g$, that is 
\begin{equation*}
\begin{array}{ll}
f_{k}(x,y) & =(f\circ g)(x,y)=f(g(x,y))=f(x,ky) \\ 
&  \\ 
& =\left\{ 
\begin{array}{l}
L{n}\dfrac{(x-c+a)}{ky}\text{ \ \ \ \ \ \ \ \ if \ }(x,y)\in L_{kac}
\\ 
L{n}(ky)\text{ \ \ \ \ \ \ \ \ \ \ \ \ \ \ \ \ \ \ if }(x,y)\in L_{p}\text{.%
}%
\end{array}%
\right. 
\end{array}%
\end{equation*}

\begin{proposition}
Every $(\mathbb{H}_{k},d,f_{k})$ is a metric geometry.
\end{proposition}

\begin{proof}
Since $d_{k}=d$, the axioms i), ii) and iii) are satisfied.

To show that 
\begin{equation*}
f_{k}:L_{kac}\rightarrow \mathbb{R}\ni \text{ }f_{k}(x,y)=\text{ }L{n}%
\dfrac{(x-c+a)}{ky}
\end{equation*}%
is one-to-one onto one must show that for every $t\in \mathbb{R}$ there is
only one pair of $(x,y)$ which satisfies 
\begin{equation*}
(x-c)^{2}+k^{2}y^{2}=a^{2},\text{ }y>0\text{ for }f(x,y)=t\text{.}
\end{equation*}%
If $f(x,y)=t$ then 
\begin{equation*}
f(x,y)=L{n}\dfrac{(x-c+a)}{ky}=t\Rightarrow \dfrac{(x-c+a)}{ky}=e^{t}\text{.%
}
\end{equation*}%
Thus, 
\begin{equation*}
\begin{array}{ll}
e^{-t} & =\dfrac{ky}{x-c+a}=\dfrac{ky}{x-c+a}\text{ }.\text{ }\dfrac{x-c-a}{%
x-c-a}=\dfrac{ky(x-c-a)}{(x-c)^{2}-a^{2}} \\ 
&  \\ 
& =\dfrac{ky(x-c-a)}{-k^{2}y^{2}}=\dfrac{x-c-a}{-ky}%
\end{array}%
\end{equation*}%
and 
\begin{equation*}
e^{t}+e^{-t}=\dfrac{x-c+a}{ky}-\dfrac{x-c-a}{ky}=\dfrac{2a}{ky}\Rightarrow y=%
\frac{a}{k\cosh t},
\end{equation*}%
and 
\begin{equation*}
e^{t}-e^{-t}=\dfrac{x-c+a}{ky}+\dfrac{x-c-a}{ky}=\dfrac{2(x-c)}{ky}\text{.}
\end{equation*}%
Thus%
\begin{equation*}
\tanh t=\dfrac{2(x-c)}{ky}.\dfrac{ky}{2a}=\dfrac{x-c}{a}\Rightarrow
x=c+a\tanh t\text{.}
\end{equation*}%
That is, the only possible solution to $f_{k}(x,y)=t$ is $x=c+a\tanh t$ and $%
y=a$ $/$ $k\cosh t$.

Similarly to show that 
\begin{equation*}
f_{k}:L_{p}\rightarrow \mathbb{R}\ni \text{ }f_{k}(x,y)=\text{ }L{n}(ky)
\end{equation*}%
is one-to-one onto, let $t\in \mathbb{R}\ni $ $f_{k}(x,y)=t$. Then 
\begin{equation*}
L{n}(ky)=t\text{ }\Rightarrow ky=e^{t}\Rightarrow y=e^{t}/k
\end{equation*}%
and%
\begin{equation*}
x\in L_{p}\Rightarrow x=p\text{.}
\end{equation*}%
Consequently only solution is $(p,e^{t}/k)$ and $f_{k}$ is 1-1 onto.

Finaly, if $x_{1}\neq x_{2}$ then $P,Q\in L_{kac}$ and 
\begin{equation*}
\begin{array}{ll}
\left\vert f_{k}(P)-f_{k}(Q)\right\vert & =\left\vert L{n}\dfrac{x_{1}-c+a}{%
ky_{1}}-L{n}\dfrac{x_{2}-c+a}{ky_{2}}\right\vert \\ 
& =\left\vert L{n}\left( \dfrac{x_{1}-c+a}{y_{1}}\text{ }/\text{ }\dfrac{%
x_{2}-c+a}{y_{2}}\right) \right\vert =d(P,Q)%
\end{array}%
\end{equation*}%
Thus $f_{k}$ is a ruler for $L_{kac}$.

If if $x_{1}=x_{2}$ then $P,Q\in L_{p}$ and 
\begin{equation*}
\begin{array}{ll}
\left\vert f_{k}(P)-f_{k}(Q)\right\vert & =\left\vert
L{n}(ky_{1})-L{n}(ky_{2})\right\vert \\ 
& =\left\vert L{n}\dfrac{ky_{1}}{ky_{2}}\right\vert =\left\vert L{n}\dfrac{%
y_{1}}{y_{2}}\right\vert \\ 
& =d(P,Q)%
\end{array}%
\end{equation*}%
and $f_{k}$ is a ruler for $L_{p}$.
\end{proof}

Every distance on an incidence geometry doesn't give a metric geometry. The
ruler postulate is very strong condition to place an incidence geometry,
which allows us to investigate further properties. If a metric geometry
satisfies \textit{the plane separation axiom} (PSA) below, then it is called 
\textit{Pasch Geometry}.

\textbf{PSA.} For every line $l$ in $\mathcal{L}$, there are two subsets $%
H_{1}$ and $H_{2}$ of $\mathcal{P}$ (called half planes determined by $l$)
such that

i) $H_{1}\cup $ $H_{2}=\mathcal{P}-l$ ($\mathcal{P}$ with $l$ removed)

ii) $H_{1}$ and $H_{2}$ are disjoint and each is convex,

iii) If $A\in H_{1}$ and $B\in $ $H_{2}$, then $AB\cap l\neq \varnothing $.

\begin{proposition}
Every $(\mathbb{H}_{k},d,f_{k})$ satisfies plane separation axiom.
\end{proposition}

\begin{proof}
Let $l$ be a $h-$line in $\mathbb{H}_{k}$. Let $H_{1}$ and $H_{2}$ be the
half planes determined by $l$ such that 
\begin{equation*}
H_{1}=\left\{ (x,y)\in \mathcal{P}\text{ }\mid x>p\text{ }\right\}
\end{equation*}%
\begin{equation*}
H_{2}=\left\{ (x,y)\in \mathcal{P}\text{ }\mid x<p\text{ }\right\} \text{ \
if \ \ }l=L_{p}
\end{equation*}%
and 
\begin{equation*}
H_{1}=\left\{ (x,y)\in \mathcal{P}\text{ }\mid (x-c)^{2}+k^{2}y^{2}>a^{2}%
\text{ }\right\}
\end{equation*}%
\begin{equation*}
H_{2}=\left\{ (x,y)\in \mathcal{P}\text{ }\mid
(x-c)^{2}+k^{2}y^{2}<a^{2}\right\} \text{ \ if \ \ }l=L_{kac}\text{.}
\end{equation*}%
It can be easily seen that $H_{1}$ and $H_{2}$ are disjoint and i) and iii)
are satisfied. Furthermore convexity is a result of the fact that two $h-$%
lines meet at most one point in $\mathbb{H}_{k}$ (see Fig 4).

\clearpage

\begin{figure}[ht]
	\centering
		\includegraphics[width=1\textwidth]{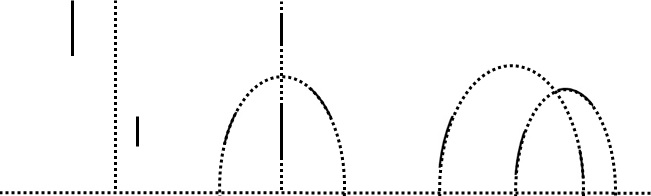}
	\text{ Fig.4 }
\end{figure}

It is very well known that a metric geometry satisfies PSA iff it satisfies
Pasch Axiom: \textit{A line which intersects one side of a triangle must
intersect one of the other two sides}.
\end{proof}

\textit{Is }$H_{k}$\textit{\ a protructor geometry?}

Cleraly, one can use Euclidean angle measure in a generalized Poincar\'{e}
plane since it is a subset of the Euclidean plane and since its lines are
defined in terms of Euclidean lines and ellipses. Here, the basic idea is to
replace the elliptical rays that make up the angle by Euclidean rays that
are tangents to the given elliptical rays. Thus without going into details
we can give the following :

\begin{proposition}
Every $(\mathbb{H}_{k},d,f_{k})$ with Euclidean angle measure is a
protractor geometry.
\end{proposition}

A\textit{\ neutral} (or absolute) \textit{geometry} is a protractor geometry
which satisfies Side-Angle-Side axiom (SAS). It is not difficult to deduce
the result that \textit{every }$\mathbb{H}_{k}$ \textit{is a neutral geometry%
} if its angle measure is defined a similar way to that of the original
Poincar\'{e} plane.

\textit{Open Questions}

In this paper it has been shown that generalized upper half-plane with the
Poincar\'{e} distance function gives a metric geometry $\mathbb{H}_{k}$ . Is
it possible to find a distance function distinct from that of the Poincar%
\'{e} distance for $\mathbb{H}_{k}$, using the hyperbolic distance of the
elliptical arcs?

For this, one has to use the elliptic integral $\int \sqrt{k^{2}+\cot t}dt$%
, $k\neq 1$.

also, there are some problems that are worth studying. Think, what they
are!\bigskip 

\textbf{References}

\begin{enumerate}
\item Anderson, J.W., Hyperbolic Geometry Springer-Verlag, London Berlin
Heidelberg (1964).

\item Birkhoff, G., A Set of Postulates for Plane Geometry, Based on Scale
and Protractor, Annals of Math., 33 (1932), 329-345.

\item \c{C}olako\u{g}lu, H.B. - Kaya, R., A generalization of some
well-known distances and related isometries, Math. Commun. 16 (2011), 21-35.

\item Greenberg, M.J., Euclidean and Non-Euclidean Geometries, Development
and History, W.H. Freeman and Company (1993).

\item Kaya, R. - Geli\c{s}gen, \"{O}., - Ekmek\c{c}i, S., - Bayar, A., On
The group of Isometries of the Plane with Generalized Absolute Value Metric,
Rocky Mountain J. of Math. 39, (2009), 591-603.

\item Krause, F.G., Taxicab Geometry, An Adventure in Non-Euclidean
Geometry, Dover Publications, Inc., New York (1986).

\item Millman, R.S., - Parker, G. D., Geometry, A Metric Approach with
Models, Springer-Verlag, New York Berlin Heilberg, (1991).

\item Stahl, S., The Poincar\'{e} Half-Plane, A Gateway to Modern Geometry,
Jones and Bartlett Publihers, Boston London Inc. (1993).
\end{enumerate}

\end{document}